\documentclass[a4paper]{amsart} 
\usepackage{amssymb} 
\usepackage[all]{xy}
\SelectTips{cm}{}

\usepackage{hyperref}
\calclayout

\title{The Gabriel--Roiter filtration of the Ziegler spectrum}

\thanks{Version from November 30, 2011.}

\author{Henning Krause}
\address{Henning Krause\\ Fakult\"at f\"ur Mathematik\\
Universit\"at Bielefeld\\ D-33501 Bielefeld\\ Germany.}
\email{hkrause@math.uni-bielefeld.de}
\author{Mike Prest}
\address{Mike Prest\\School of Mathematics\\
Alan Turing Building\\
University of Manchester\\
Manchester M13 9PL\\United Kingdom.} 
\email{mprest@manchester.ac.uk}

\newtheorem{lem}{Lemma}[section]
\newtheorem{prop}[lem]{Proposition}
\newtheorem{cor}[lem]{Corollary}
\newtheorem{thm}[lem]{Theorem}

\newtheorem*{Thm}{Theorem}

\theoremstyle{remark}

\theoremstyle{definition}
\newtheorem{exm}[lem]{Example}

\numberwithin{equation}{section}

\newcommand{\smatrix}[1]{\left[\begin{smallmatrix}#1\end{smallmatrix}\right]}

\renewcommand{\mod}{\operatorname{\mathsf{mod}}\nolimits}

\newcommand{\sub}{\operatorname{\mathsf{sub}}\nolimits}

\newcommand{\id}{\operatorname{id}\nolimits}

\newcommand{\Mod}{\operatorname{\mathsf{Mod}}\nolimits}
\newcommand{\Mon}{\operatorname{\mathsf{Mon}}\nolimits}
\newcommand{\Ind}{\operatorname{\mathsf{Ind}}\nolimits}

\newcommand{\Def}{\operatorname{\mathsf{Def}}\nolimits}
\newcommand{\Cl}{\operatorname{\mathsf{Cl}}\nolimits}
\newcommand{\Zg}{\operatorname{\mathsf{Zg}}\nolimits}
\newcommand{\GR}{\mathsf{GR}}

\newcommand{\lto}{\longrightarrow}

\def\li{\varinjlim}
\def\lp{\varprojlim}

\def\a{\alpha}
\def\b{\beta}

\def\m{\mu}

\def\la{\lambda}

\def\La{\Lambda}

\def\C{{\mathsf C}}
\def\D{{\mathsf D}}

\def\Sc{{\mathsf S}}

\def\T{{\mathsf T}}
\def\U{{\mathsf U}}
\def\V{{\mathsf V}}

\def\bbN{{\mathbb N}}

\begin{document}

\begin{abstract}
  Inclusion preserving maps from modules over an Artin algebra to
  complete partially ordered sets are studied.  This yields a
  filtration of the Ziegler spectrum which is indexed by all
  Gabriel--Roiter measures. Another application is a compactness
  result for the set of subcategories of finitely presented modules
  that are closed under submodules.
\end{abstract}

\maketitle

\section{Introduction}

Let $A$ be an Artin algebra.  We work in the category $\Mod A$ of all
$A$-modules and $\mod A$ denotes the full subcategory consisting of
all finitely presented $A$-modules. 

In this paper we combine two concepts from representation theory which
have the following in common: they are powerful but also technically
involved. Our motivation is to understand invariants of
representations which reflect the inclusion relation. Thus we study
maps $f\colon\Mod A\to\Sc$ where $\Sc$ is a partially ordered set and
for each pair $X,Y$ of $A$-modules
\[X\subseteq Y \quad \implies \quad f(X)\le f(Y).\]  The Gabriel--Roiter measure
$\mu\colon\Mod A\to \mathbf{2}^\bbN$ in the sense of Ringel
\cite{Ri2005} is an example of particular importance. Here,
$\mathbf{2}^\bbN$ denotes the power set of the set of natural numbers,
endowed with the lexicographical order.

In a recent paper \cite{R}, Ringel used the Gabriel--Roiter measure to
establish the following somewhat surprising result.  Here, an additive
subcategory of $\mod A$ is said to be of \emph{infinite type} if it
contains infinitely many non-isomorphic indecomposable objects.

\begin{Thm}[Ringel]
Each submodule closed additive subcategory of $\mod A$ that is 
of infinite type contains one which is minimal among all submodule
closed additive subcategories of infinite type.\qed
\end{Thm}

We give a new proof of this result which involves the Ziegler spectrum
of $A$ and uses its compactness \cite{Z}. A further analysis then
leads to a filtration of the Ziegler spectrum which is indexed by the
totally ordered set $\{\mu(X)\mid X\in\Mod A\}$ consisting of all
Gabriel--Roiter measures.

\section{From modules to partially ordered sets}

In this section we study maps from the category of $A$-modules to
complete partially ordered sets.  From a categorical point of view
this means we consider the subcategory $\Mon A$ of $\Mod A$ where the
objects are the $A$-modules and the morphisms between two modules are
the $A$-linear monomorphisms.  Then we study functors $\Mon A\to \Sc$
where $\Sc$ is a partially ordered set, viewed as a category having at
most one morphism between any two objects.

\subsection*{Submodule closed subcategories}
Let $\mathsf S(\mod A)$ denote the set of full additive subcategories
of $\mod A$ that are closed under submodules. This set is partially
ordered by inclusion and in fact complete.

Recall that a partially ordered set $\Sc$ is \emph{complete} if every
subset $\U$ of $\Sc$ has a supremum, which then is denoted by
\[\sup\U=\bigvee_{x\in\U} x.\]
Note that the supremum can be expressed as an infimum:
\[\sup\U=\inf \{y\in\Sc\mid x\le y\text{ for all }x\in\U\}.\]

Given an $A$-module $X$, let $\sub X$ denote the full subcategory of
$\mod A$ consisting of all $A$-modules that are submodules of finite
direct sums of copies of $X$.

\begin{prop}\label{pr:universal}
  The map $\Mod A\to\Sc(\mod A)$ taking a module $X$ to $\sub X$ is
  the universal map $f\colon\Mod A\to \Sc$ to a complete partially
  ordered set $\Sc$ satisfying
\begin{enumerate}
\item $f(X)\le f(Y)$ for $X\subseteq Y$ in $\Mod A$;
\item $f(X\oplus Y)=f(X)\vee f(Y)$ for $X,Y$ in $\Mod A$;
\item $f(\bigcup_\a X_\a)=\bigvee_\a f(X_\a)$ for every directed union
  $\bigcup_\a X_\a$ in $\Mod A$.
\end{enumerate}
More precisely, given such a map $f\colon\Mod A\to \Sc$, there exists
a unique map $\bar f\colon\Sc(\mod A)\to\Sc$ satisfying $f(X)=\bar
f(\sub X)$ for all $X\in\Mod A$. The map $\bar f$ is order preserving
and \[\bar f(\bigvee_\a\C_\a)=\bigvee_\a \bar f(\C_\a)\] for every set
of elements $\C_\a\in\Sc(\mod A)$.
\end{prop}
\begin{proof}
  It is clear that the assignment $X\mapsto\sub X$ satisfies
  (1)--(3). Now fix an arbitrary map $f\colon\Mod A\to \Sc$ with these
  properties. Then $\sub X\subseteq \sub Y$ implies that $X$ is a
  submodule of a finite direct sum of copies of $Y$, and therefore
  $f(X)\le f(Y)$.  Thus $\bar f \colon\Sc(\mod A)\to\Sc$ taking $\sub
  X$ to $f(X)$ is well-defined and order preserving. Note that any
  $\C$ in $\Sc(\mod A)$ is of the form $\C=\sub X_\C$, where
  $X_\C=\bigoplus_{X\in\C}X$. Finally, we compute
\begin{equation*}
  \bar f(\bigvee_\a\C_\a)=\bar f(\sub\bigoplus_\a X_{\C_\a})= 
  f(\bigoplus_\a X_{\C_\a})= \bigvee_\a f(X_{\C_\a})=\bigvee_\a \bar f(\C_\a).\qedhere
\end{equation*}
\end{proof}

\subsection*{The Ziegler spectrum}
We write $\Ind A$ for the set of isomorphism classes of indecomposable
pure-injective $A$-modules.  A subset of $\Ind A$ is \emph{Ziegler
  closed} if it is of the form $\C\cap\Ind A$ for some definable
subcategory $\C\subseteq\Mod A$. Following \cite{CB}, a subcategory is
\emph{definable} if it is closed under filtered colimits, products and
pure submodules.  The Ziegler closed subsets provide the closed
subsets of a topology on $\Ind A$; see \cite{CB,Z}.  For each class
$\C$ of $A$-modules, we denote by $\Def\C$ the smallest definable
subcategory containing $\C$ and let $\Zg\C=\Def\C\cap\Ind A$. Note
that
\begin{equation}\label{eq:Zg}
\Zg\Def\C=\Zg\C\quad\text{and}\quad
\Def\Zg\C=\Def\C.
\end{equation}
The first equality is clear from the definition; for the second one,
see \cite[\S2.3]{K} or \cite[Corollary~6.9]{Z}.

Given an additive subcategory $\C$ of $\mod A$, let $\varinjlim\C$
denote the full subcategory consisting of all $A$-modules that are
filtered colimits of modules in $\C$.

\begin{prop}\label{pr:def}
  Let $\C$ be an additive subcategory of $\Mod A$ that is closed under
  submodules.  Then
\[\Def\C=\varinjlim(\C\cap\mod A)=\{X\in\Mod A\mid\sub X\subseteq\C\}.\]
\end{prop}

\begin{proof}
  We may assume that $\C\subseteq\mod A$; the general case is then an
  immediate consequence. For each $X\in\mod A$, let $X\to X_\C$ denote
  the universal morphism to an object of $\C$. This is an epimorphism,
  since $\C$ is closed under submodules; take $X_\C=X/U$ where $U$
  denotes the minimal submodule with $X/U\in\C$. An $A$-module $Y$
  belongs to $\varinjlim\C$ if and only if each morphism $X\to Y$ with
  $X$ finitely presented factors through the morphism $X\to X_\C$; see
  \cite[Proposition~2.1]{L}. It follows that an $A$-module $Y$ belongs
  to $\varinjlim\C$ if and only if every finitely presented submodule
  belongs to $\C$. From the same description, it is easily seen that
  $\varinjlim\C$ is closed under filtered colimits and products. Thus
  $\varinjlim\C=\Def\C$.
\end{proof}

\begin{cor}\label{co:Zg-sub}
Let $\C\subseteq\Mod A$ be a full additive subcategory closed under
submodules. Then 
\[\Zg\C=\{X\in\Ind A\mid\sub X\subseteq\C\}.\]
For a set of submodule closed full additive subcategories
$\C_\a\subseteq\Mod A$, one has \[\Zg(\bigcap_\a \C_\a)=\bigcap_\a
\Zg\C_\a.\]
\end{cor}
\begin{proof}
  The first part is clear from the preceding proposition.  Now let
  $\C=\bigcap_\a \C_\a$.  We need to check that
  $\Zg\C\supseteq\bigcap_\a \Zg\C_\a$ while the other inclusion is clear.
  Fix a module $Y$ in $\bigcap_\a \Zg\C_\a$. A finitely presented
  submodule of $Y$ belongs to $\C_\a$ for all $\a$, and therefore it
  belongs to $\C$. Thus $Y$ is in $\Zg\C$.
\end{proof}

The following example shows that in the preceding corollary the
assumption on each $\C_\a$ to be submodule closed is necessary.

\begin{exm}
  Let $A$ be a tame hereditary algebra. Given any tube $\C$ of the
  AR-quiver, $\Zg\C$ contains the unique generic $A$-module
  \cite[Corollary~8.6]{K1998}. Thus we have for two diferent tubes
  $\C_1,\C_2$ that $\Zg\C_1\cap\Zg\C_2\neq\varnothing$, while
  $\C_1\cap\C_2 =\varnothing$.
\end{exm}

For each class $\C$ of $A$-modules, let $\sub\C$ denote the full
subcategory consisting of all finitely presented submodules of finite
direct sums of modules in $\C$.

\begin{cor}\label{co:sub}
Let $\C$ be a class of $A$-modules. Then
\[\sub\C=\sub\Zg\C=\sub\Def\C.\]
\end{cor}
\begin{proof}
  We apply Proposition~\ref{pr:def} and get
\[\sub\C\subseteq\sub\Def\C\subseteq\sub\varinjlim\sub\C=\sub\C.\]
Combining this identity with \eqref{eq:Zg} gives
\[\sub\Zg\C=\sub\Def\Zg\C=\sub\Def\C=\sub\C.\qedhere\]
\end{proof}

\begin{cor}\label{co:Zg}
  Let $f\colon\Mod A\to\Sc$ be a map to a complete partially ordered
  set $\Sc$ satisfying the conditions (1)--(3) from
  Proposition~\ref{pr:universal}. Then
  \[f(X)=\bigvee_{Y\in\Zg X}f(Y) \qquad\text{for all }X\in\Mod A.\] 
\end{cor}
\begin{proof}
From Corollary~\ref{co:sub} one has  \[\sub X=\sub \Zg
X=\bigvee_{Y\in\Zg X}\sub Y.\]
Using the map $\bar f\colon \Sc(\mod A)\to\Sc$ from  Proposition~\ref{pr:universal},
one gets
\[f(X)=\bar f(\sub X)=\bar f \big(\bigvee_{Y\in\Zg X}\sub Y\big)=
\bigvee_{Y\in\Zg X}\bar f(\sub Y) =\bigvee_{Y\in\Zg X} f(Y).\qedhere\]
\end{proof}

\section{The Gabriel--Roiter filtration} 

In this section we study a specific inclusion preserving map $\Mod
A\to\Sc$, namely the Gabriel--Roiter measure. This map refines the
usual length function $\Mod A\to \bbN$ and has the
additional property that the set $\Sc$ is totally ordered.

\subsection*{The Gabriel--Roiter measure} 
Let $\bbN=\{1,2,3,\ldots\}$ and denote by 
$\mathbf{2}^\bbN$ the set of all subsets of $\bbN$.
We view this as a partially ordered set via
the \emph{lexicographical order}, given by
\[I\le J\quad\Longleftrightarrow\quad\inf(J\setminus I)\le
\inf(I\setminus J)\qquad\text{for } I,J\in \mathbf{2}^\bbN.\]
Note that  $\mathbf{2}^\bbN$ is totally ordered and complete.

Given an $A$-module $X$ of finite length, let $\ell(X)$ denote the
length of a composition series. Following \cite{Ga1973,Ri2005}, the
\emph{Gabriel--Roiter measure} of an $A$-module $X$ is \[\mu
(X)=\bigvee_{X_1\subsetneq \ldots\subsetneq X_r\subseteq
  X}\{\ell(X_1),\ldots,\ell(X_r)\},\] where $X_1\subsetneq
\ldots\subsetneq X_r\subseteq X$ runs through all finite chains of
submodules such that each $X_i$ is indecomposable and of finite
length. For a class $\C$ of $A$-modules, we write
\[\mu (\C)=\bigvee_{X\in\C}\mu(X).\]
The basic properties of the Gabriel--Roiter measure are summarised in
the following statement. Note that these are precisely the properties
appearing in Proposition~\ref{pr:universal}.

\begin{prop}\label{pr:GR}
  Let $X,Y$ be $A$-modules. Then
\begin{enumerate}
\item $\mu(X)\le\mu(Y)$ if $X\subseteq Y$;
\item $\mu(X)=\bigvee_\a \mu(X_\a)$ for every directed union
  $X=\bigcup_\a X_\a$;
\item $\mu(X\oplus Y)=\mu(X)\vee\mu(Y)$.
\end{enumerate}
\end{prop}
\begin{proof}
(1) and (2) are clear from the definition of $\mu$. (3) holds for
finitely presented $A$-modules by \cite[Corollary~5.3]{Ga1973}. To
prove the general case, write  $X=\bigcup_\a X_\a$ and  $Y=\bigcup_\b
Y_\b$ as directed unions of finitely presented modules. Then  \[X\oplus
Y=\bigcup_{(\a,\b)} X_\a\oplus Y_\b\] and therefore
\begin{align*}
\mu(X\oplus Y) &=\bigvee_{(\a,\b)} \mu(X_\a\oplus Y_\b)\\
&=\bigvee_{(\a,\b)} \mu(X_\a)\vee\mu(Y_\b)\\
&= \big(\bigvee_\a \mu(X_\a)\big)\vee\big(\bigvee_\b\mu(Y_\b)\big)\\
&=\mu(X)\vee\mu(Y).\qedhere
\end{align*}
\end{proof}

\begin{cor}\label{co:def}
Let $\C$ be a class of $A$-modules. Then 
\[\mu(\C) =\mu(\sub\C) =\mu(\Zg\C)=\mu(\Def\C).\]
\end{cor}
\begin{proof}
  The first identity follows from Proposition~\ref{pr:GR}.  The rest
  then follows by Corollary~\ref{co:sub}.
\end{proof}

It seems to be an interesting question to ask, whether each element
$I=\mu(X)$ in the image of $\mu\colon\Mod A\to\mathbf{2}^\bbN$ is of
the form $I=\mu(Y)$ for some indecomposable pure-injective $A$-module
$Y$.

\subsection*{The Gabriel--Roiter filtration}

The following proposition yields a collection of (not necessarily
distinct) Ziegler closed subsets of $\Ind A$ which is indexed by the
elements of $\mathbf 2^\bbN$. For each $I\in\mathbf 2^\bbN$, set
\[\Zg I=\{X\in\Ind A\mid\mu(X)\le I\}\quad\text{and}\quad\sub
I=\{X\in\mod A\mid\mu(X)\le I\}.\]

\begin{prop}\label{pr:filtr}
  Let $I\in\mathbf 2^\bbN$.  
\begin{enumerate}
\item The set $\Zg I$ is Ziegler
  closed and the subcategory $\sub I$ is additive and submodule
  closed.
\item  If $I=\mu(X)$ for some $A$-module $X$, then
\[\mu(\Zg I)=I \quad\text{and}\quad \mu(\sub I)=I .\]
\item For each subset $\U\subseteq\Ind A$, one has
\[\mu(\U)\le I\quad\iff\quad \U\subseteq \Zg I.\]
\item For each subcategory $\C\subseteq\mod A$, one has
\[\mu(\C)\le I\quad\iff\quad \C\subseteq \sub I.\]
\end{enumerate}
\end{prop}
\begin{proof}
  The $A$-modules $X$ satisfying $\mu(X)\le I$ form an additive
  subcategory of $\Mod A$ that is closed under submodules, by
  Proposition~\ref{pr:GR}. In fact, these modules form a definable
  subcategory, by Proposition~\ref{pr:def}, and therefore $\Zg I$ is
  Ziegler closed. The rest is clear from the definitions of $\Zg I$
  and $\sub I$.
\end{proof}

We shorten our notation and set $\V_I=\Zg I$ for each  $I\in\mathbf 2^\bbN$.
\begin{cor}
  There is a filtration $(\V_I)_{I\in\mathbf 2^\bbN}$ of $\Ind A$
  consisting of Ziegler closed subsets such that the following holds:
\begin{enumerate}
\item $\V_I\subseteq \V_J$ for all $I\le J$ in $\mathbf 2^\bbN$;
\item $\V_{\inf \Sc}=\bigcap_{I\in\Sc} \V_{I}$ for all
  $\Sc\subseteq\mathbf 2^\bbN$;
\item $\mu(\V_I)\le I$ for all $I\in\mathbf 2^\bbN$, and equality
  holds if and only if $I=\mu(X)$ for some $A$-module $X$.\qed
\end{enumerate}
\end{cor}

\subsection*{The partially ordered set of Ziegler closed sets}

We denote by $\Cl(\Ind A)$ the set of Ziegler closed subsets of $\Ind
A$; they form a complete partially ordered set.  Corollary~\ref{co:Zg}
says that the map taking an $A$-module $X$ to $\Zg X$ is universal in
the sense that any map $f\colon\Mod A\to\Sc$ to a complete partially
ordered set satisfying the conditions (1)--(3) from
Proposition~\ref{pr:universal} satisfies \[f(X)=\bigvee_{Y\in\Zg
  X}f(Y).\] The basic examples of such assignments are $X\mapsto\sub
X$ and $X\mapsto\mu (X)$.  This yields the following diagram:
\[
\xymatrix@C2pc@R=2pc{&&\Mod A\ar[d]^-\sub\ar[dll]_-\Zg\ar[drr]^-\mu\\
\Cl(\Ind A) \ar@<.5ex>[rr]^-\sub&&\mathsf S(\mod
A)\ar@<.5ex>[ll]^-\Zg
\ar@<.5ex>[rr]^-\mu&&\mathbf 2^\bbN\ar@<.5ex>[ll]^-\sub
}\]
Here, we write
\[
\xymatrix@C2pc@R=2pc{\Sc \ar@<.5ex>[rr]^-f&&\T \ar@<.5ex>[ll]^-g}\]
for an \emph{adjoint pair} of morphisms between partially ordered sets
which means that 
\[f(x)\le y \quad\iff\quad x\le g(y) \qquad\text{for all
}x\in\Sc,\,y\in\T.\] 
The adjointness of the pair $(\sub,\Zg)$ follows from
Corollary~\ref{co:Zg-sub}; for $(\mu,\sub)$ it follows from Proposition~\ref{pr:filtr}.

We say that a morphism $f\colon\Sc\to\T$ is a \emph{quotient map} if $f$
induces an isomorphism $\Sc/_\sim\to\T$, where $x\sim y$ iff
$f(x)=f(y)$ for $x,y\in\Sc$. An equivalent condition is that
$fg=\id_\T$; see \cite[Proposition~I.1.3]{GZ}.

Let us denote by $\GR(A)$ the image of $\mu\colon\Mod A\to \mathbf
2^\bbN$. This is a complete partially ordered set. 

\begin{prop} 
The morphisms
\[\sub\colon\Cl(\Ind A)\lto\Sc(\mod
A)\quad\text{and}\quad\mu\colon\Sc(\mod A)\lto\GR(A)\]
are quotient maps.
\end{prop}
\begin{proof}
  We have $\sub\Zg \C=\C$ for each $\C\in\Sc(\mod A)$, by
  Corollary~\ref{co:sub}. On the other hand, $\mu (\sub I)= I$ for
  each $I\in\GR(A)$, by Proposition~\ref{pr:filtr}.
\end{proof}

Given a pair of Ziegler closed subsets $\U,\V$ of $\Ind A$, when is
$\sub\U=\sub\V$? This amounts to computing $\Zg\sub\U$, since
\[\sub\U=\sub\V \quad\iff\quad \Zg\sub\U=\Zg\sub\V.\]
Note that 
\[\V\subseteq \Zg\sub \V\]
holds automatically; we  describe when equality holds.

\begin{prop}
Let $\C$ be a definable subcategory of $\Mod A$ and $\V=\C\cap\Ind A$
the corresponding Ziegler closed set. Then the following are
equivalent:
\begin{enumerate}
\item $\C$ is closed under submodules;
\item $\V$ is closed under submodules: $X\in\V$, $Y\in\Ind A$, and
  $Y\subseteq X^n$ for some $n\in\bbN$ implies $Y\in\V$;
\item $\V=\Zg\sub\V$.
\end{enumerate}
\end{prop}
\begin{proof} 
(1) $\Rightarrow$ (2): Clear.

(2) $\Rightarrow$ (3): That $\V$ is closed under submodules implies
$\Def\V=\Def\sub\V$. Using \eqref{eq:Zg} then gives
\[\V=\Zg\Def\V=\Zg\Def\sub\V=\Zg\sub\V.\]

(3) $\Rightarrow$ (1): The equality in (3) yields
\[\C=\Def\V=\Def\Zg\sub\V=\Def\sub\V=\Def\sub\Def\V=\Def\sub\C.\]
Here, \eqref{eq:Zg} and Corollary~\ref{co:sub} are used. The equality
$\C=\Def\sub\C$ implies that $\C$ is closed under submodules.
\end{proof}

\subsection*{The Kronecker algebra}
Let $\La=\smatrix{k&k^2\\ 0&k}$ be the Kronecker algebra over an
algebraically closed field $k$.  A complete list of the
indecomposables in $\mod\La$ is given by the preprojectives $P_n$, the
regulars $R_n(\la)$, and the preinjectives $Q_n$; see
\cite[Thm.~VIII.7.5]{ARS}. More precisely,
\[\Ind\La\cap\mod\La=\{P_n\mid n\in\bbN\}\cup\{R_n(\la)\mid
n\in\bbN,\,\la\in\mathbb P^1(k)\}\cup\{Q_n\mid n\in\bbN\},\] and the
inclusion order is described by the following Hasse diagram.
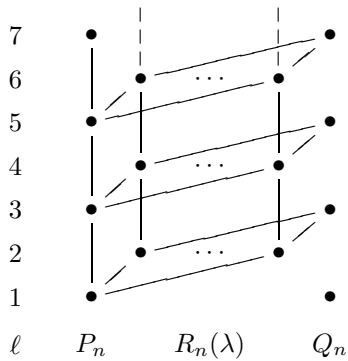
\[\xymatrix@=0.4em{
&&\ar@{--}[d]&\ar@{--}[dd]&&\ar@{--}[dd]\\
7&&\bullet\ar@{-}[dd]&&&&\bullet\ar@{-}[dl]\ar@{-}[dlll]\\
6&&&\bullet\ar@{-}[dd]\ar@{-}[dl]&\cdots&\bullet\ar@{-}[dd]\ar@{-}[dlll]\\
5&&\bullet\ar@{-}[dd]&&&&\bullet\ar@{-}[dl]\ar@{-}[dlll]\\
4&&&\bullet\ar@{-}[dd]\ar@{-}[dl]&\cdots&\bullet\ar@{-}[dd]\ar@{-}[dlll]\\
3&&\bullet\ar@{-}[dd]&&&&\bullet\ar@{-}[dl]\ar@{-}[dlll]\\
2&&&\bullet\ar@{-}[dl]&\cdots&\bullet\ar@{-}[dlll]\\
1&&\bullet&&&&\bullet\\ \ell&&P_n&&R_n(\la)&&Q_n }\] 
From this, one computes
\begin{align*}
\mu(P_n)&=\{1,3,5,\ldots,2n-1\}\\
\mu(R_n)&=\{1,2,4,\ldots,2n\}\\
\mu(Q_n)&=\{1,2,4,\ldots ,2n-2,2n-1\}
\end{align*}
where the Gabriel--Roiter measure of $R_n=R_n(\la)$ does not depend on
$\la$. This gives the following order:
\[{\scriptstyle\m(Q_1)=\m(P_1)<\m(P_2)<\m(P_3)<\;\;\ldots\;\;<\m(R_1)<\m(R_2)<\m(R_3)<\;\;\ldots\;\:<\m(Q_4)<\m(Q_3)<\m(Q_2)}\]
The indecomposable pure-injective $\La$-modules which are not finitely
presented are the Pr\"ufer modules $R_\infty(\la)=\li R_n(\la)$, the
adic modules $\widehat R(\la) =\lp R_n(\la)$, and the generic
module $G$; see \cite{P1998,R1998}. Thus
\[\Ind \La\setminus\mod \La=\{R_\infty(\la),\widehat R(\la)\mid
\la\in\mathbb P^1(k)\}\cup\{G\}.\] 
Now one computes
\begin{gather*}
\mu(\widehat R(\la))=\mu(G)=\{1,3,5,7,\ldots\}=\bigvee_{n\ge 1}\mu(P_n)\\
\mu(R_\infty(\la))=\{1,2,4,6,\ldots\}=\bigvee_{n\ge 1}\mu(R_n)=\bigwedge_{n\ge 1}\mu(Q_n)
\end{gather*}
and this completes the list of values of the Gabriel--Roiter measure;
see also \cite[Appendix~B]{Ri2005}. Note that this yields the
description of the Gabriel--Roiter filtration of $\Ind\La$.

\section{Compactness}

The collection of submodule closed additive subcategories of $\mod A$
enjoys a compactness property which we discuss in this section.  A
consequence is the existence of minimal submodule closed subcategories
of infinite type. This is a somewhat surprising result from a recent
article of Ringel \cite{R}. Note that the proof given here is quite
different from Ringel's.  He uses the Gabriel--Roiter measure, while
the compactness result is derived from the compactness of the Ziegler
spectrum.

Let $\C$ be an additive subcategory of $\mod A$ which is closed under
direct summands. We say that $\C$ is of \emph{finite type} if $\C$
contains only finitely many pairwise non-isomorphic indecomposable
modules. Note that a submodule closed subcategory $\C$ is of finite
type if and only if the set
\[\{\D\in\Sc(\mod A)\mid \D\subseteq\C\}\]
is finite.

\begin{thm}\label{th:compact}
Let $(\C_\a)_{\a\in\La}$ be a collection of additive subcategories
$\C_\a\subseteq\mod A$ that are submodule closed. If
$\C=\bigcap_{\a\in\La}\C_\a$ is of finite type, then there is a finite
subset $\La'\subseteq \La$ such that $\C=\bigcap_{\a\in\La'}\C_\a$.
\end{thm}

The proof uses some properties of the Ziegler spectrum which are collected
in the following proposition. For a general introduction, we refer the reader to
\cite{K,P2009}.

\begin{prop}\label{pr:Zg}
The space $\Ind A$ has the following properties.
\begin{enumerate}
\item The space $\Ind A$ is quasi-compact.
\item For $X\in\Ind A\cap\mod A$, the subset $\{X\}$ is open.
\item An additive subcategory $\C\subseteq\mod A$ is of finite type iff $\Zg\C\subseteq\mod A$.
\end{enumerate}
\end{prop}
\begin{proof}
(1)  See \cite[Theorem~4.9]{Z} or \cite[\S2.5]{CB}.

(2) See \cite[Proposition~13.1]{P1988}.

(3) If $\C$ is of finite type, then the direct sums of modules in $\C$ form a definable subcategory; see \cite[\S2.5]{CB}. Thus $\Zg\C\subseteq\mod A$. If $\C$ is of infinite type, then part (1) and (2) imply that $\Zg\C$ contains modules which are not finitely presented.
\end{proof}

\begin{proof}[Proof of Theorem~\ref{th:compact}]
  We have $\Zg\C=\bigcap_{\a\in\La}\Zg\C_\a$ by
  Corollary~\ref{co:Zg-sub}. Using the properties of $\Ind A$
  collected in Proposition~\ref{pr:Zg}, it follows that
  $\Zg\C=\bigcap_{\a\in\La'}\Zg\C_\a$ for some finite subset
  $\La'\subseteq \La$. We have $\sub\Zg\D=\D$ for each submodule
  closed additive subcategory $\D\subseteq\mod A$, by
  Corollary~\ref{co:sub}.  Thus $\C=\bigcap_{\a\in\La'}\C_\a$.
\end{proof}

A combination of Theorem~\ref{th:compact} with Zorn's lemma gives the
following result, and  Ringel's theorem  \cite{R} mentioned in the
introduction is an immediate consequence.

\begin{cor}
  Let $\Sc$ be a set of submodule closed additive subcategories of
  $\mod A$ that is closed under forming intersections.  Then the
  subset of $\Sc$ consisting of all subcategories of infinite type is
  either empty or it has a minimal element.  \qed
\end{cor}


\begin{thebibliography}{99}
%
\bibitem{ARS} M. Auslander, I. Reiten\ and\ S. O. Smal{\o}, {\it
    Representation theory of Artin algebras}, Cambridge Studies in
  Advanced Mathematics, 36, Cambridge Univ. Press, Cambridge, 1995.
%
\bibitem{CB} W. Crawley-Boevey, Infinite-dimensional modules in the representation theory of finite-dimensional algebras, in {\it Algebras and modules, I (Trondheim, 1996)}, 29--54, CMS Conf. Proc., 23 Amer. Math. Soc., Providence, RI, 1998.
%
\bibitem{Ga1973} P. Gabriel, Indecomposable representations. II, in
  {\it Symposia Mathematica, Vol. XI (Convegno di Algebra Commutativa,
    INDAM, Rome, 1971)}, 81--104, Academic Press, London, 1973.
%
\bibitem{GZ}P. Gabriel\ and\ M. Zisman, {\it Calculus of fractions and
homotopy theory}, Springer-Verlag New York, Inc., New York, 1967.
%
\bibitem{K1998} H. Krause, Generic modules over Artin algebras, Proc. London Math. Soc. (3) {\bf 76} (1998), no.~2, 276--306.
%
\bibitem{K} H. Krause, The spectrum of a module category, Mem. Amer. Math. Soc. {\bf 149} (2001), no.~707, x+125 pp.
%
\bibitem{L} H. Lenzing, Homological transfer from finitely presented to infinite modules, in {\it Abelian group theory (Honolulu, Hawaii, 1983)}, 734--761, Lecture Notes in Math., 1006 Springer, Berlin, 1983.
%
\bibitem{P1988} M. Prest, {\it Model theory and modules}, London Mathematical Society Lecture Note Series, 130, Cambridge Univ. Press, Cambridge, 1988. 
%
\bibitem{P1998} M. Prest, Ziegler spectra of tame hereditary algebras,
  J. Algebra {\bf 207} (1998), no.~1, 146--164
%
\bibitem{P2009} M. Prest, {\it Purity, spectra and localisation}, Encyclopedia of Mathematics and its Applications, 121, Cambridge Univ. Press, Cambridge, 2009. 
%
\bibitem{R1998} C. M. Ringel, The Ziegler spectrum of a tame
  hereditary algebra, Colloq. Math. {\bf 76} (1998), no.~1, 105--115.
%
\bibitem{Ri2005} C. M. Ringel, The Gabriel-Roiter measure,
  Bull. Sci. Math. {\bf 129} (2005), no.~9, 726--748.
%
\bibitem{R} C. M. Ringel, Minimal infinite submodule-closed
  subcategories, arXiv:1009.0864v1.
%
\bibitem{Z} M. Ziegler, Model theory of modules, Ann. Pure Appl. Logic {\bf 26} (1984), no.~2, 149--213.

\end{thebibliography}
\end{document}